\documentclass[a4paper,pdftex,reqno,10pt]{elsarticle}
\usepackage{geometry}
\usepackage{graphicx}
\usepackage{multirow}
\usepackage{times}
\usepackage{amssymb,amsmath}

\def\@begintheorem#1#2{\list{}{\thm@body}%
  \item[]{\bf #1~#2.}\quad\it\ignorespaces}
\def\@opargbegintheorem#1#2#3{\list{}{\thm@body}%
  \item[]{\bf #1~#2~\ifrembrks #3\global\rembrksfalse\else (#3)\fi.}%
  \quad\it\ignorespaces}
\def\@endtheorem{\endlist}
\newtheorem{theorem}{Theorem}[section]

\newtheorem{corollary}[theorem]{Corollary}
\newtheorem{lemma}[theorem]{Lemma}

\newtheorem{definition}[theorem]{Definition}

\newproof{proof}{ Proof}

\begin{document}

\title{No finite $5$-regular matchstick graph exists}

\author[ubt]{Sascha Kurz}
\ead{sascha.kurz@uni-bayreuth.de}
\address[ubt]{Fakult\"at f\"ur Mathematik, Physik und Informatik, Universit\"at Bayreuth, Germany}

\begin{abstract}
  A graph $G=(V,E)$ is called a unit-distance graph in the plane if there is an injective embedding of $V$ in
  the plane such that every pair of adjacent vertices are at unit distance apart. If additionally the corresponding
  edges are non-crossing and all vertices have the same degree $r$ we talk of a regular matchstick graph. Due to
  Euler's polyhedron formula we have $r\le 5$. The smallest known $4$-regular matchstick graph is the so called
  Harborth graph consisting of $52$ vertices. In this article we prove that no finite $5$-regular matchstick graph
  exists.
\end{abstract}


\maketitle

\section{Prologue}
\noindent
One of the possibly best known problems in combinatorial geometry asks how often the same distance can occur among
$n$ points in the plane. Via scaling we can assume that the most frequent distance has length $1$. Given any set $P$ of points in the plane, we can define the so called unit-distance graph in the plane, connecting two elements of $P$ by an edge if their distance is one. The known bounds for the maximum number $u(n)$ of edges of a unit-distance graph in the plane, see e.~g.{} \cite{1086.52001}, are given by
\[
  \Omega\!\left(ne^{\frac{c\log n}{\log\log n}}\right)\le u(n)\le O\!\left(n^{\frac{4}{3}}\right).
\]
For $n\le 14$ the exact numbers of $u(n)$ were determined in \cite{schade}, see also \cite{1086.52001}.

If we additionally require that the edges are non-crossing, then we obtain another class of geometrical and combinatorial objects:

\begin{definition}
  A \textbf{matchstick graph} $\mathcal{M}$ consists of a graph $G=(V,E)$
  and an injective embedding $g:V\rightarrow\mathbb{R}^2$ in the plane which fulfill the following conditions:
  \begin{itemize}
    \item[(1)] $G$ is a planar (simple) graph,
    \item[(2)] for all edges $\{i,j\}\in E$ we have $\left\Vert g(i),g(j)\right\Vert_2=1$, where $\Vert x,y\Vert_2$
               denotes the Euclidean distance between the points $x$ and $y$,
    \item[(3)] $g(i)\neq g(j)$ for $i\neq j$, and
    \item[(4)] if $\left\{i_1,j_1\right\},\left\{i_2,j_2\right\}\in E$ for pairwise different vertices
               $i_1,i_2,j_1,j_2\in V$ then the line segments $\overline{g\left(i_1\right)g\left(j_1\right)}$
               and $\overline{g\left(i_2\right)
               g\left(j_2\right)}$ do not have a common point.
  \end{itemize}
\end{definition}

\noindent
For matchstick graphs the known bounds for the maximum number $\tilde{u}(n)$ of edges, see e.~g.{} \cite{1086.52001}, are given by
\[
  \left\lfloor 3n-\sqrt{12n-3}\right\rfloor\le \tilde{u}(n)\le 3n-O\!\left(\sqrt{n}\right),
\]
where the lower bound is conjectured to be exact.

We call a matchstick graph $r$-regular if every vertex has degree $r$. In \cite{matchsticks_in_the_plane} the authors consider $r$-regular matchstick graphs with the minimum number $m(r)$ of vertices. Obviously we have $m(0)=1$, $m(1)=2$, and $m(2)=3$, corresponding to a single vertex, a single edge, and a triangle, respectively.

The determination of $m(3)$ is an entertaining amusement. At first we observe that $m(3)$ must be even since every graph contains an even number of vertices with odd degree. For $n\le 8$ vertices the set of $3$-regular planar graphs is fairly assessable, consisting of one graph of order, i.~e.{} the number of vertices, $4$, one graph of order $6$, and three graphs of order $8$. Utilizing area arguments rules out all but one graph of order $8$, so that we have $m(3)=8$. For degree $r=4$ the exact determination of $m(4)$ is unsettled so far. The smallest known example is the so called Harborth graph, see e.~g.{} \cite{gerbracht}, yielding $m(4)\le 52$.

Due to the Eulerian polyhedron formula every finite planar graph contains a vertex of degree at most five so that we have $m(r)=\infty$ for $r\ge 6$. We would like to remark that the regular triangular lattice is the unique example of an infinite matchstick graph with degree at least $6$.

For degree $5$ it is announced at several places that no finite $5$-regular matchstick graph does exist, see e.~g.{} Ivars Peterson's MathTrek ``Match Sticks in the Summer'' \cite{peterson} most likely referring to personal communication with Heiko Harborth, the discoverer of the $4$-regular matchstick graph. Indeed Heiko Harborth posed the question whether there exists a $5$-regular matchstick graph, or not at an Oberwolfach meeting and was aware of a preprint claiming the non-existence proof (personal communication). After a while he found out that there were some mistakes in the proof, so that to his knowledge there is no correct proof of the non-existence. Asking the author of this preprint about the state of this problem he replied too, that the problem is still open (at this point in time the concerning was untraceable).

Wolfram Research's MathWorld, see http://mathworld.wolfram.com/MatchstickGraph.html, refers to Erich Friedman, who himself maintains a webpage \cite{friedman} stating the non-existence of a $5$-regular matchstick graph. Asking Erich Friedman for a proof of this claim he responded
that unfortunately he lost his reference but also thinks that such a graph cannot exist.

In \cite{1063.05036} the authors list five publications where they have mentioned an unpublished proof for the non-existence of a $5$-regular matchstick graph and state that up to their knowledge the problem is open.

After the submission of this paper one of the referees came up with a very short proof for the non-existence \cite{short}. In this paper we would like to give a different, admittedly more complicated, proof using a topological equation that needs some explanation. 
So in some parts this (unpublished) paper coincides with \cite{short}.

In the meantime Heiko Harborth could recover the lost preprint and send it to the original author and myself. Curiously enough, after 
rectifying some annoying minor mistakes, the preprint turns out to be basically correct. So the author of this article took the opportunity to retype and slightly modify the original manuscript, see \cite{manuscript}. It even turned out that the underlying  method can be easily adopted to prove the following stronger result: No finite matchstick graph with minimum degree $5$ does exist. So up to now there are at least three proofs of Theorem~\ref{thm_main} whose pairwise intersections are non-empty.


We would like to mention the recent proof of the Higuchi conjecture \cite{1117.05026} -- a result of similar flavor where related techniques are applied.

\section{Main Theorem}

\begin{theorem}
  \label{thm_main}
  No finite $5$-regular matchstick graph does exist.
\end{theorem}

\noindent
In order to prove Theorem \ref{thm_main} we denote the number of $i$-gons, i.~e.{} a face consisting of $i$ vertices, of a given matchstick graph $\mathcal{M}$ by $F_i$, where we also count the outer face. In the example in Figure \ref{fig_matchstick_graph} we have $F_3=3$, $F_4=2$, $F_7=1$, and $F_i=0$ for all other $i$.

\begin{figure}[htp]
\begin{center}
\includegraphics{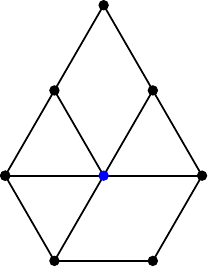}
\caption{A matchstick graph.}
\label{fig_matchstick_graph}
\end{center}
\end{figure}

\begin{lemma}
  \label{lemma_euler_sum}
  For a finite $5$-regular matchstick graph we have
  \[
    \sum_{i=3}^{\infty}(10-3i)F_i=F_3-2F_4-5F_5-8F_6-11F_7-\dots=20.
  \]
\end{lemma}
\begin{proof}
  Due to the Eulerian polyhedral formula we have $V+F-E=2$, where $V$ denotes the number of vertices, $F$ the number
  of faces, and $E$ the number of edges. The number of faces is given by
  \[
    F=\sum_{i=3}^\infty F_i.
  \]
  Since every edge is part of two faces and every vertex is part of $5$ faces we have
  \[
    2E=\sum_{i=3}^{\infty} iF_i\quad\text{and}\quad
    5V=\sum_{i=3}^{\infty} iF_i.
  \]
  Inserting yields the proposed formula.
\end{proof}

\begin{definition}
  The \textbf{face set} $f(v)$ of a vertex $v$ contained in a planar graph is a multiset containing the number of corners
  of the adjacent faces.
\end{definition}

In the example in Figure \ref{fig_matchstick_graph} we have $f(v)=\{3,3,3,4,4\}$. For a $5$-regular matchstick graph the face sets $f(v)$ of the vertices have cardinality $5$, i.~e. there exist integers $a_{v,1},\dots,a_{v,5}\ge 3$ with $f(v)=\left\{a_{v,1},\dots,a_{v,5}\right\}$. Due to the angle sum of $2\pi$ at a vertex we have $a_{v,i}\ge 4$ for at least one index $i$ for every vertex $v\in V$.

\begin{definition}
  For an integer $a\ge 3$ we set $c(a)=\frac{10-3a}{a}$ and use this to define the \textbf{contribution} $c(v)$ of
  a vertex with face set $f(v)=\left\{a_{v,1},\dots,a_{v,5}\right\}$ as
  \[
    c(v)=c\left(\left\{a_{v,1},\dots,a_{v,5}\right\}\right)=\sum_{i=1}^5c\!\left(a_{v,i}\right)=
    \sum_{i=1}^5\frac{10-3a_{v,i}}{a_{v,i}}.
  \]
  For $U\subseteq V$ we set $c(U)=\sum\limits_{u\in U}c(u)$.
\end{definition}
 
Now we can relate this definition to Lemma \ref{lemma_euler_sum} and state two easy lemmas.

\begin{lemma}
  \label{lemma_20}
  If $\mathcal{M}$ is a (finite) $5$-regular matchstick graph we have $c(V)=20$.
\end{lemma}
\begin{proof}
  For $j\ge 3$ we have $\left|\left\{a_{v,i}=j\mid v\in V,\, 1\le i\le 5\right\}\right|=jF_j$. Applying this
  to the definition of $c(V)$ yields the left hand side of the formula of Lemma \ref{lemma_euler_sum}.
\end{proof}

We would like to remark that for an infinite $5$-regular matchstick graph we would have $c(V)=0$.

\begin{lemma}
  \label{lemma_non_negative}
  The face sets with non-negative contribution are given by
  \begin{eqnarray*}
    c\left(\left\{3,3,3,3,4\right\}\right)=\frac{5}{6}, && c\left(\left\{3,3,3,4,4\right\}\right)=0,\\
    c\left(\left\{3,3,3,3,5\right\}\right)=\frac{1}{3}, && c\left(\left\{3,3,3,3,6\right\}\right)=0.
  \end{eqnarray*}
\end{lemma}

\noindent
Our strategy for proving Theorem \ref{thm_main} will be to partition the vertex set $V$ into subsets each having non-positive contribution, which yields a contradiction to Lemma \ref{lemma_20}.

To determine a suitable partition of $V$ into subsets for a given matchstick graph $\mathcal{M}$ we consider the set $\mathcal{TQ}$ of triangles and quadrangles with internal angles in $\left\{\frac{1}{3}\pi,\frac{2}{3}\pi\right\}$ and define an equivalence relation $\sim$ on $\mathcal{TQ}$. Whenever there exist $x,y\in\mathcal{TQ}$ sharing an edge we require $x\sim y$. We complete this relation to an equivalence relation by taking the transitive closure.

Since we prove that no finite $5$-regular matchstick graphs exists it is impossible to give examples to illustrate our definitions. Therefore we define them (mostly) for matchstick graphs $\mathcal{M}$ where the degrees of its vertices
are at most $5$.

\begin{definition}
  For a given (possible infinite) matchstick graph $\mathcal{M}$ with vertex degrees at most $5$ we call an
  equivalence class $x\sim =\left\{y\in\mathcal{TQ}\mid y\sim x\right\}$ of the above defined equivalence
  relation a $\mathcal{TQ}$-class.
\end{definition}

So a $\mathcal{TQ}$-class is an edge-connected union of triangles and quadrangles whose vertices are situated on a suitable common regular triangular lattice. In Figure \ref{fig_honeycomb_components} we have depicted an example, where we have marked the $\mathcal{TQ}$-classes by different face colors.

\begin{figure}[htp]
\begin{center}
\includegraphics{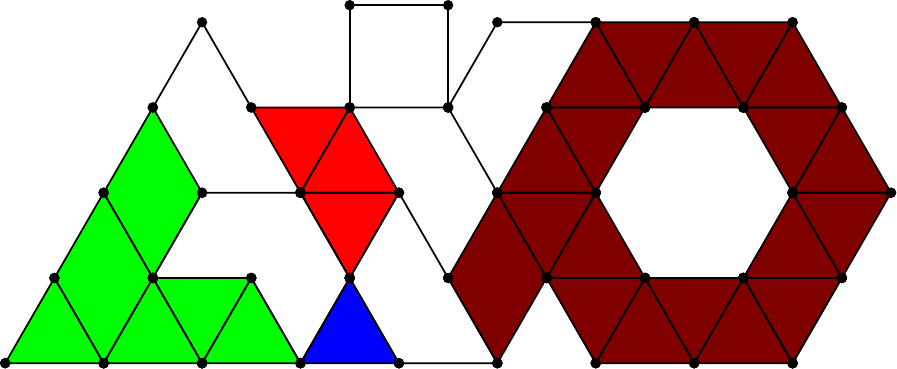}
\caption{$\mathcal{TQ}$-classes of a matchstick graph.}
\label{fig_honeycomb_components}
\end{center}
\end{figure}

\noindent
In the next lemma we summarize some easy facts on $\mathcal{TQ}$-classes.

\begin{lemma}
  Let $\mathcal{B}$ be a $\mathcal{TQ}$-class, then the following holds:
  \begin{itemize}
   \item[(1)] The faces of $\mathcal{B}$ are edge-to-edge connected, meaning that the dual graph is connected.
   \item[(2)] The vertices of $\mathcal{B}$ are situated on a suitable regular triangular lattice.
  \end{itemize}
\end{lemma}

\noindent
For brevity we associate with a $\mathcal{TQ}$-class $\mathcal{B}$ the set of its vertices $V(\mathcal{B})$ and the set of its edges $E(\mathcal{B})$ so that we can utilize the notation $c(\mathcal{B})$ for the contribution of the vertex set of $\mathcal{B}$. Now our aim is to show that the contribution $c\!\left(\mathcal{B}\right)$ for every $\mathcal{TQ}$-class $\mathcal{B}$ of a $5$-regular matchstick graph is at most zero. Since some vertices may be contained in more than one $\mathcal{TQ}$-class the corresponding vertex sets can not be used directly to partition the vertex set $V$. Instead we choose the union $\mathcal{C}=\cup_i\mathcal{B}_i$ of all $\mathcal{TQ}$-classes. Here every vertex $v\in V$ occurs at most once in $\mathcal{C}$, nevertheless it may be contained in several $\mathcal{TQ}$-classes $\mathcal{B}_i$. In order to prove $c(\mathcal{C})\le 0$ we have to perform some bookkeeping during the proof of $c(\mathcal{B})\le 0$, which will be the topic of the next section. Once we haven proven this we can state:

\bigskip

\noindent
\textbf{Proof of the main theorem.} Let $\mathcal{M}$ be a finite $5$-regular matchstick graph and $\mathcal{C}$ be the union of all $\mathcal{TQ}$-classes of $\mathcal{M}$. Using Corollary \ref{cor_non_negativ_all} and the fact that all vertices with positive contribution are contained in a $\mathcal{TQ}$-class, see Lemma \ref{lemma_non_negative}, we conclude
\[
  c(V)\,\,\,\,=\,\,\,\,c(\mathcal{C})+\sum_{v\in V\backslash\mathcal{C}} c(v)\,\,\,\,\le\,\,\,\, 0+
  \sum_{v\in V\backslash\mathcal{C}}0\,\,\,\,=\,\,\,\,0.
\]
This contradicts Lemma \ref{lemma_20}.
\hfill{$\square$}

\section{Contribution of $\mathcal{TQ}$-classes of $5$-regular matchstick graphs}

\noindent
In this section we want to study the contribution $c(\mathcal{B})$ of a finite $\mathcal{TQ}$-class $\mathcal{B}$ in a $5$-regular matchstick graph $\mathcal{M}$. Since we show that no finite $5$-regular matchstick graph exists $\mathcal{M}$ has to be infinite in this context. Nevertheless there may be some finite $\mathcal{TQ}$-classes in $\mathcal{M}$.

But since we are only interested in finite matchstick graphs we introduce another concept and consider incomplete finite parts of $5$-regular matchstick graphs.

\begin{definition}
  Let $\mathcal{M}$ be a finite matchstick graph with maximum vertex degree at most $5$ and $\mathcal{B}$ a
  $\mathcal{TQ}$-class of $\mathcal{M}$ which induces a planar graph $\mathcal{G}$ containing the vertices, edges,
  and faces of $\mathcal{B}$. We call a vertex $v$ in $\mathcal{G}$ (or $\mathcal{B}$) an \textbf{inner vertex} if
  all faces being adjacent to $v$ in $\mathcal{G}$ are contained in $\mathcal{B}$. All other vertices $v$ in
  $\mathcal{G}$ (or $\mathcal{B}$) are called \textbf{outer vertices}. Now we can call the $\mathcal{TQ}$-class
  $\mathcal{B}$ \textbf{prospective $\mathbf{5}$-regular} exactly if all inner vertices of $\mathcal{G}$ have
  degree $5$ and all outer vertices of $\mathcal{G}$ have degree at most $5$.
\end{definition}

\begin{figure}[htp]
\begin{center}
\includegraphics{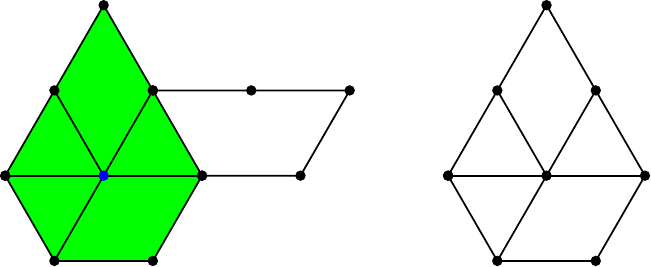}
\caption{A prospective $5$-regular $\mathcal{TQ}$-class $\mathcal{B}$
of a matchstick graph $\mathcal{M}$
and its induced planar graph $\mathcal{G}$
.}
\label{fig_prospective}
\end{center}
\end{figure}

\noindent
In Figure \ref{fig_prospective} we have depicted a matchstick graph with maximum vertex degree at most $5$ on the left hand side. The faces of the, in this case uniquely, contained $\mathcal{TQ}$-class $\mathcal{B}$ are filled with green color. If we build up a planar graph out of the vertices, edges, and faces of $\mathcal{B}$ we obtain the right hand side of Figure \ref{fig_prospective}. In $\mathcal{G}$ vertex $v$ is the only vertex which is not adjacent to the outer
face. Thus $v$ is an inner vertex in $\mathcal{B}$ and the remaining $7$ vertices of $\mathcal{B}$ are outer vertices.

\medskip

Since every $\mathcal{TQ}$-class $\mathcal{B}$ of a given finite $5$-regular matchstick graph $\mathcal{M}$ is prospective $5$-regular, we now study prospective $5$-regular $\mathcal{TQ}$-classes. Therefore we want to specify them by some parameters $\sigma$, $k$, $\tau$, $b_1$ and $b_2$. (For brevity we forego to use notations as $\sigma(\mathcal{B})$, \dots whenever the corresponding $\mathcal{TQ}$-class is evident from the context.)

\medskip

The parameter $\sigma$ stands for the contribution of the faces of $\mathcal{B}$. To become more precisely we have to introduce a new technical notation. A given vertex $v$ is adjacent to five faces denoted as $(v,1),\dots,(v,5)$, where face $(v,i)$ is an $a_{v,i}$-gon. Since there exist pairs of vertices $u\neq v$ and indices $1\le i,j\le 5$ with $(v,i)=(u,j)$ we introduce the notation $[v,i]$ addressing the arc of face $(v,i)$ at vertex $v$. The contribution of such a face arc $[v,i]$ is defined to be $c\!\left(a_{v,i}\right)$. If $\mathcal{A}$ is the set of all face arcs $[v,i]$ where the face $(v,i)$ is contained in $\mathcal{B}$, then $\sigma$ is given by $\sum\limits_{\alpha\in\mathcal{A}}c(\alpha)$.

\medskip

Let us look at the four $\mathcal{TQ}$-classes of Figure \ref{fig_honeycomb_components} (it is easy to check that they are all prospective $5$-regular). The blue $\mathcal{TQ}$-class consists of a single triangle. So it has three \textit{triangle}-arcs and a contribution of $\sigma_{\text{blue}}=\frac{1}{3}+\frac{1}{3}+\frac{1}{3}=1$. It is easy to figure out that in general a triangle contributes $1$ and a quadrangle contributes $-2$ to $\sigma$ so that we have $\sigma_{\text{red}}=3$, $\sigma_{\text{green}}=0$, and $\sigma_{\text{brown}}=16$.

We have already remarked that every $e\in E$ edge of a planar graph is contained in two faces. Whenever both faces are contained in a given $\mathcal{TQ}$-class $\mathcal{B}$ we say that $e$ is an inner edge. If only one face is contained in $\mathcal{B}$ we call $e$ an outer edge. By $k$ we count the outer edges of a $\mathcal{TQ}$-class $\mathcal{B}$. In our example we have $k_{\text{blue}}=3$, $k_{\text{red}}=5$, $k_{\text{green}}=10$, and $k_{\text{brown}}=20$.

If we consider the vertices and edges of a prospective $5$-regular $\mathcal{TQ}$-class as a subgraph, then every vertex $v$ has a degree $\delta(v)$ at most $5$. By $\tau(v)$ we denote the gap $5-\delta(v)$ (free valencies) and by $\tau(\mathcal{B})$ the sum over all $\tau(v)$ where $v$ is in $\mathcal{B}$. So in our example we have $\tau_{\text{blue}}=9$, $\tau_{\text{red}}=11$ , $\tau_{\text{green}}=20$ , and $\tau_{\text{brown}}=22$.

In a $5$-regular planar graph every vertex $v$ is contained in exactly five different faces $(v,i)$. Whenever all five faces are contained in a given $\mathcal{TQ}$-class $\mathcal{B}$ we say that $v$ in $\mathcal{B}$ is an inner vertex. If the number of faces $(v,i)$ which are contained in $\mathcal{B}$ is between one and four we call $v$ an outer vertex. This coincides with our previous definition of inner and outer vertices of a prospective $5$-regular $\mathcal{TQ}$-class. Clearly we have $\tau(v)=0$ for all inner vertices $v$. If $e=\{v,u\}$ is an outer edge in $\mathcal{B}$ then $v$ and $u$ are outer vertices of $\mathcal{B}$. For the other direction we have that for an outer vertex $v$ in $\mathcal{B}$ there exist at least two outer edges $e_1$, $e_2$ in $\mathcal{B}$ being adjacent to $v$. More precisely the number of outer edges being adjacent to $v$ is either $2$ or $4$.

\begin{figure}[htp]
\begin{center}
\includegraphics{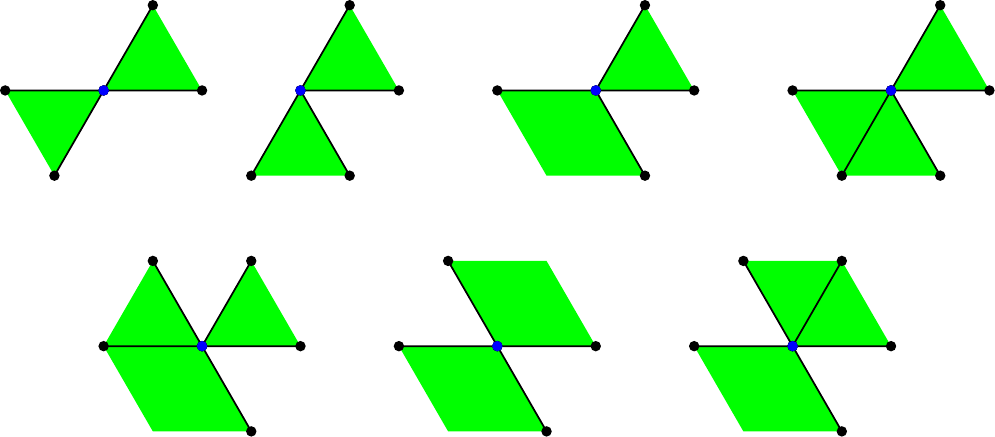}
\caption{Vertices of special type.}
\label{fig_special_situation}
\end{center}
\end{figure}

For a given vertex $v$ in $\mathcal{B}$ we consider the faces $(v,i)$ which are contained in a given $\mathcal{TQ}$-class $\mathcal{B}$ as vertices of a graph $H$. Two vertices of $H$ are connected via an edge whenever the corresponding faces have a common edge in $\mathcal{M}$. If the graph $H$ is connected we call $v$ of normal type. Otherwise we say that vertex $v$ is of special type. In Figure \ref{fig_special_situation} we depict all cases of vertices of special type up to symmetry. Here vertex $v$ is marked blue and the faces of $\mathcal{B}$ are marked green. Alternatively we could also define a vertex of special type as an outer vertex of a prospective $5$-regular $\mathcal{TQ}$-class which is adjacent to exactly $4$ outer edges of $\mathcal{B}$. We would like to remark that from a local point of view at a vertex of special type it seems that the faces of the $\mathcal{TQ}$-class are not edge connected.

\medskip

By $b_1$ we count the number of vertices $v$ in $\mathcal{B}$ of special type.  In our example of Figure \ref{fig_honeycomb_components} no vertex of special type exists and we have $b_1=0$. In Figure \ref{fig_honeycomb_components_2} we have depicted another example of a $\mathcal{TQ}$-class, where we have $b_1=1$ -- the blue vertex.

\begin{figure}[htp]
\begin{center}
\includegraphics{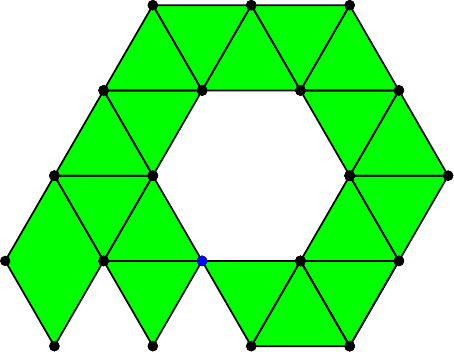}
\caption{Another $\mathcal{TQ}$-class of a matchstick graph.}
\label{fig_honeycomb_components_2}
\end{center}
\end{figure}

If a face $f\in\mathcal{B}$ of a $\mathcal{TQ}$-class is a quadrangle then there are two inner angles of $\frac{2\pi}{3}$ and two inner angles of $\frac{\pi}{3}$. By $b_2$ we count the number of inner angles of $\frac{2\pi}{3}$ where the corresponding vertex $v$ is an outer vertex. It may happen that an outer vertex
$v$ is part of two quadrangles in $\mathcal{B}$ each having an inner angle of $\frac{2}{3}\pi$ at $v$. Thus
for a prospective $5$-regular $\mathcal{TQ}$-class $\mathcal{B}$ we may write $b_2(v)\in\{0,1,2\}$ for the
number of inner angles of $\frac{2\pi}{3}$ at an outer vertex $v$. For inner vertices we set $b_2(v)=0$ so
that we can set $b_2=\sum_{v\in V(\mathcal{B})}b_2(v)$. In the example
of Figure \ref{fig_honeycomb_components} we have $b_2=4$ in the green class, $b_2=2$ in the brown class,
and $b_2=0$ in the red and blue class. We remark that due to an angle sum of $2\pi$ there is exactly one 
face with an inner angle of $\frac{2\pi}{3}$ at every inner vertex of $\mathcal{B}$.

\medskip

Now we want to prove the equation
\begin{equation}
  \label{eqn_paramater}
  \sigma-k+\frac{\tau-k}{3}+\frac{5}{3}b_1+\frac{5}{3}b_2=0.
\end{equation}
relating the parameters $\sigma$, $k$, $\tau$, $b_1$ and $b_2$. This equation even holds for more general objects
than prospective $5$-regular $\mathcal{TQ}$-classes. A $\mathcal{TR}$-class $\mathcal{B}$ is a union of
triangles and quadrangles with inner angles in $\left\{\frac{\pi}{3},\frac{2\pi}{3}\right\}$ on a regular triangular grid
such that each inner vertex has degree $5$ and each outer vertex has degree at most $5$. Clearly we can transfer the
definitions of the parameters $\sigma$, $k$, $\tau$, $b_1$ and $b_2$ from (prospective $5$-regular) $\mathcal{TQ}$-classes to $\mathcal{TR}$-classes. The crucial difference between $\mathcal{TQ}$-classes and $\mathcal{TR}$-classes is that the dual graph of a $\mathcal{TR}$-class (considered as a planar graph) is not connected in general.

\begin{figure}[htp]
\begin{center}
\includegraphics{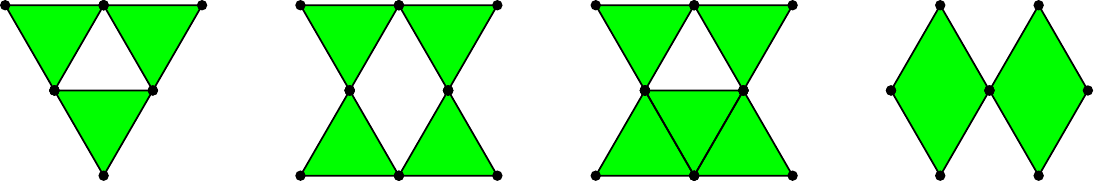}
\caption{Examples of $\mathcal{TR}$-classes.}
\label{fig_tr_classes}
\end{center}
\end{figure}

In Figure \ref{fig_tr_classes} we depict some extraordinary examples of $\mathcal{TR}$-classes to demonstrate the whole variety being covered by the definition of a $\mathcal{TR}$-class. Here the faces of a $\mathcal{TR}$-class $\mathcal{B}$ are filled in green. Those faces which are not contained in $\mathcal{B}$ are left in white. In the left example of Figure \ref{fig_tr_classes} we have $\sigma=3$, $k=9$, $\tau=12$, $b_1=3$, and $b_2=0$. In the second example we have $\sigma=4$, $k=12$, $\tau=16$, $b_1=4$, and $b_2=0$. In the third example of Figure \ref{fig_tr_classes}
we have $\sigma=5$, $k=11$, $\tau=14$, $b_1=3$, and $b_2=0$. For the central vertex $v$ of the right example of Figure \ref{fig_tr_classes} we have $b_2(v)=2$. Additionally $v$ is of special type. The parameters are given by $\sigma=-4$, $k=8$, $\tau=19$, $b_1=1$, and $b_2=4$. Thus in all four cases Equation (\ref{eqn_paramater}) is valid. If we consider those four examples as a single example on the same triangular grid again Equation (\ref{eqn_paramater}) is valid.

\begin{lemma}
  \label{lemma_parameter}
  For a $\mathcal{TR}$-class $\mathcal{B}$ Equation~(\ref{eqn_paramater}) is valid.
\end{lemma}
\begin{proof}
  We prove by induction on the number of faces of $\mathcal{B}$. For a single triangle we have $\sigma=1$, $k=3$,
  $\tau=9$, and $b_1=b_2=0$. Thus the left hand side of (\ref{eqn_paramater}) sums to zero. For a quadrangle we
  have $\sigma=-2$, $k=4$, $\tau=12$, $b_1=0$, and $b_2=2$, again summing up to zero in (\ref{eqn_paramater}). We
  remark that Equation~(\ref{eqn_paramater}) is also valid for an empty $\mathcal{TR}$-class.

  For the induction step we now
  assume that Equality (\ref{eqn_paramater}) is valid and we show that it remains valid after adding another
  single triangle or quadrangle to $\mathcal{B}$. Here we only consider those additions which do not alter
  faces of $\mathcal{B}$ (one might think of inserting an edge into a quadrangle of $\mathcal{B}$ resulting in
  two triangles and destroying the initial quadrangle). In general it might happen that adding the edges of one
  new face produces a second new face in one step. An example arises by adding an edge into a quadrangle which
  is not contained in $\mathcal{B}$, see the second and the third example of Figure \ref{fig_tr_classes}.
  Here, in one step we only add one of the resulting triangles to $\mathcal{B}$. If the other triangle should also
  be added to $\mathcal{B}$ we do this in a second step where we only have a new face and no new vertices or edges.

  \begin{figure}[htp]
\begin{center}
\includegraphics{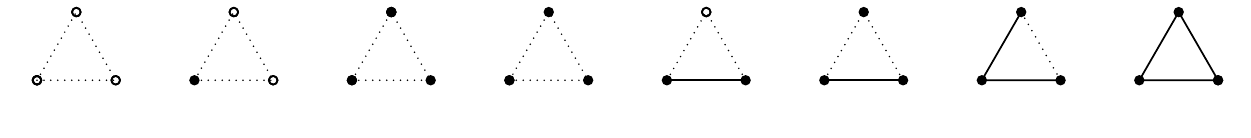}
\caption{Adding a triangle.}
\label{fig_adding_a_triangle}
\end{center}
\end{figure}

  Adding a triangle results in eight (main) cases, see Figure \ref{fig_adding_a_triangle}, and adding a quadrangle
  results in nineteen (main) cases, see Figure \ref{fig_adding_a_quadrangle}.  In the different (main) cases we
  depict the new edges by dotted lines, the new vertices by empty circles, and the old edges and vertices in black.
  We have a closer look on the change of the
  parameters. If $p$ is the parameter of $\mathcal{B}$ and $p'$ the corresponding parameter after adding a triangle
  (or a quadrangle), then we denote the change by $\Delta(p):=p'-p$. In order to prove the induction step it
  suffices to verify
  \begin{equation}
    \label{eqn_paramater_delta}
    \Delta(\sigma)-\Delta(k)+\frac{\Delta(\tau)-\Delta(k)}{3}+\frac{5}{3}\Delta\!\left(b_1\right)+
    \frac{5}{3}\Delta\!\left(b_2\right)= 0.
  \end{equation}

  We remark that all black edges in the (main) cases of Figure \ref{fig_adding_a_triangle} and Figure
  \ref{fig_adding_a_quadrangle} are outer edges before adding the new triangle or quadrangle. After adding the
  new face the black edges become inner edges and the dotted edges become outer edges.

  \begin{table}[!ht]
  \begin{center}
  \begin{tabular}{r|r|r|r|r}
    (main) case & $\Delta(\sigma)$ & $\Delta(k)$ & $\Delta(\tau)$ & $\Delta\!\left(b_1\right)+\Delta\!\left(b_2\right)$\\
    \hline
    (1) & 1 &  3 &  9 &  0\\
    (2) & 1 &  3 &  4 &  1\\
    (3) & 1 &  3 & -1 &  2\\
    (4) & 1 &  3 & -6 &  3\\
    (5) & 1 &  1 &  1 &  0\\
    (6) & 1 &  1 & -4 &  1\\
    (7) & 1 & -1 & -2 & -1\\
    (8) & 1 & -3 &  0 & -3\\
  \end{tabular}
  \caption{$\Delta(\cdot)$-values for the eight (main) cases of Figure \ref{fig_adding_a_triangle}.}
  \label{table_delta_triangles}
  \end{center}
  \end{table}

  In order to shorten our calculation and case distinction we have a look at the situation from a vertex point
  of view. For a fixed vertex $a$ (in Figure \ref{fig_adding_a_triangle}) up to symmetry there are four different
  (vertex) cases: If $a$ is a new (empty) vertex then also the two adjacent edges of the new triangle have to be
  dotted. Otherwise if $a$ is an old (black) vertex the number of dotted adjacent edges of the new triangle can be two,
  one, or zero. These four possibilities show up as (main) cases (1), (2), (5), and (7) in Figure
  \ref{fig_adding_a_triangle} and we denote them as (vertex) cases (i), (ii), (iii), and (iv). Clearly, also for
  vertex $b$ and vertex $v$ we have the same four possibilities. So we restrict our considerations on vertex $a$.

  \begin{figure}[htp]
\begin{center}
\includegraphics{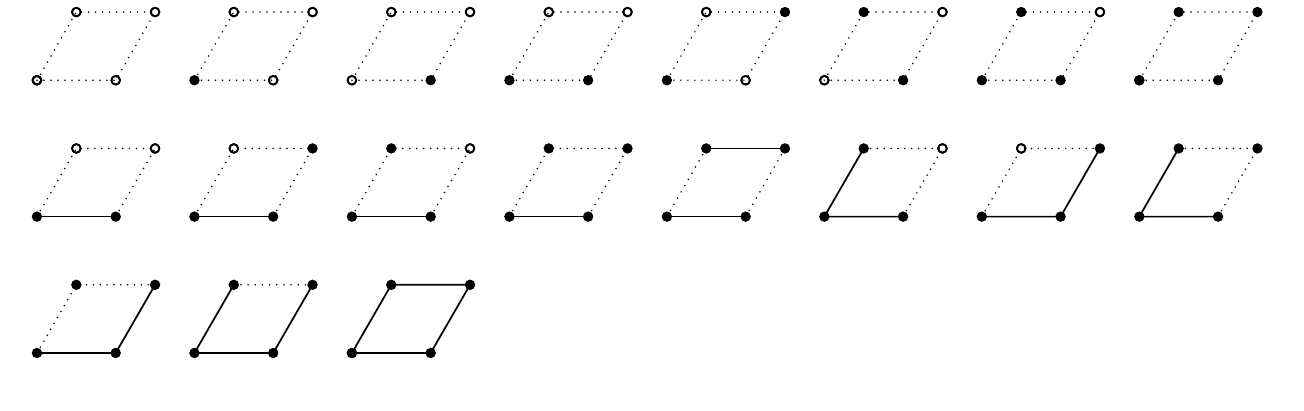}
\caption{Adding a quadrangle.}
\label{fig_adding_a_quadrangle}
\end{center}
\end{figure}

  By $a_1$ we denote the vertex $a$ before the addition of the new face and by $a_2$ we denote the vertex $a$ after
  the addition of the new face.
  In (vertex) case (i), occurring in (main) case (1), vertex $a_1$ does not exist and $a_2$ is an outer vertex.
  Obviously $a_2$ is of normal type and $b_2\!\left(a_2\right)=0$. In (vertex) case (ii), occurring in (main) case
  (2) both $a_1$ and $a_2$ are outer vertices. We have $b_2\!\left(a_1\right)
  =b_2\!\left(a_2\right)$. Since $a$ is adjacent to two red edges of the new face vertex $a_2$ is of special type
  and $a_1$ is of normal type, so that the $b_1$-count increases by one. In (vertex) case (iii) 
  both $a_1$ and $a_2$ are outer vertices. So $b_2\!\left(a_1\right)=b_2\!\left(a_2\right)$ and
  $a_1$ is of special type if and only if $a_2$ is of special type. Thus the $b_1$- and the $b_2$-counts
  do not change. In (vertex) case (iv) 
  two different things could happen at vertex $a$. At first we mention that $a_1$ is an outer
  vertex. If also $a_2$ is an outer vertex then $a_1$ is of special type and $a_2$ is of normal type. In this
  situation we have $b_2\!\left(a_1\right)=b_2\!\left(a_2\right)$. The other possibility is that $a_2$ is an inner
  vertex. Since the degree of $a_2$ is $5$ this is only possible if $b_2\!\left(a_1\right)=1$ and if $a_1$ is of
  normal type. And since $a_2$ is an inner vertex it is of normal type and we have $b_2\!\left(a_2\right)=0$. Thus
  the sum of the $b_1$-count and the $b_2$-count decreases by one.

  With the above we are able to give the $\Delta(\cdot)$-values for the (main) cases of Figure
  \ref{fig_adding_a_triangle} in Table \ref{table_delta_triangles}. It is easy to check that in all eight (main) cases
  Equation~(\ref{eqn_paramater_delta}) is valid.

\medskip

\begin{table}[ht]
  \begin{center}
  \begin{tabular}{r|r|r|r|r}
    (main) case & $\Delta(\sigma)$ & $\Delta(k)$ & $\Delta(\tau)$ & $\Delta\!\left(b_1\right)+\Delta\!\left(b_2\right)$\\
    \hline
     (9) & -2 &  4 & 12 &  2 \\
    (10) & -2 &  4 &  7 &  3 \\
    (11) & -2 &  4 &  7 &  3 \\
    (12) & -2 &  4 &  2 &  4 \\
    (13) & -2 &  4 &  2 &  4 \\
    (14) & -2 &  4 &  2 &  4 \\
    (15) & -2 &  4 & -3 &  5 \\
    (16) & -2 &  4 & -8 &  6 \\
    (17) & -2 &  2 &  4 &  2 \\
    (18) & -2 &  2 & -1 &  3 \\
    (19) & -2 &  2 & -1 &  3 \\
    (20) & -2 &  2 & -6 &  4 \\
    (21) & -2 &  0 & -4 &  2 \\
    (22) & -2 &  0 &  1 &  1 \\
    (23) & -2 &  0 &  1 &  1 \\
    (24) & -2 &  0 & -4 &  2 \\
    (25) & -2 &  0 & -4 &  2 \\
    (26) & -2 & -2 & -2 &  0 \\
    (27) & -2 & -4 &  0 & -2 \\
  \end{tabular}
  \caption{$\Delta(\cdot)$-values for the nineteen (main) cases of Figure \ref{fig_adding_a_quadrangle}.}
  \label{table_delta_quadrangles}
  \end{center}
  \end{table}

  Next we can deal with the nineteen (main) cases of Figure \ref{fig_adding_a_quadrangle}. In every (main) case
  the situation of vertex $a$ and vertex $w$ has an equivalent in Figure \ref{fig_adding_a_triangle}. So up to
  symmetry we have to consider the situation at vertex $v$. Again there are four possibilities to consider showing
  up in (main) cases (9), (14), (21), and (27) of Figure \ref{fig_adding_a_quadrangle}. We denote the corresponding
  (vertex) cases by (v), (vi), (vii), and (viii), respectively. Similarly as before we denote the vertex $v$ before
  the addition of the quadrangle by $v_1$ and afterwards by $v_2$.

  \medskip

  In (vertex) case (v), occurring in (main) case (9), $v_1$ does not exist and $v_2$ is of normal type and we
  have $b_2\!\left(v_2\right)=1$ so that the sum of the $b_1$- and the $b_2$-count increases by one. In (vertex)
  case (vi) 
  both $v_1$ and $v_2$ are outer vertices.
  Vertex $v_2$ is of special type and $v_1$ has to be of normal type. For the $b_2$-value we have
  $b_2\!\left(v_2\right)=b_2\!\left(v_1\right)+1$ so that the sum of the $b_1$- and the $b_2$-count increases by
  two. In (vertex) case (vii) 
  both $v_1$ and $v_2$ are outer vertices. Vertex $v_1$ is of special type if and only if $v_2$
  is of special type. Since $b_2\!\left(v_2\right)=b_2\!\left(v_1\right)+1$ the sum
  of the $b_1$- and the $b_2$-count increases by one. In (vertex) case (viii) 
  If $v_2$ is an inner vertex than $v_1$, $v_2$ are of normal type, $b_2\!\left(v_1\right)=0$, and
  $b_2\!\left(v_2\right)=0$. If $v_2$
  is an outer vertex then $v_1$ is of special type, $v_2$ is of normal type, and
  $b_2\!\left(v_2\right)=b_2\!\left(v_1\right)+1$. Thus in both cases the sum of the $b_1$- and
  the $b_2$-count does not change.

  \medskip

  With the above we are able to give the $\Delta(\cdot)$-values for the (main) cases of Figure
  \ref{fig_adding_a_quadrangle} in Table \ref{table_delta_quadrangles}. It is easy to check that in all
  nineteen (main) cases Equation~(\ref{eqn_paramater_delta}) is valid.
\end{proof}

\begin{corollary}
  \label{cor_parameter}
  For a (finite) prospective $5$-regular $\mathcal{TQ}$-class of a matchstick graph we have
  \[
    \sigma-k+\frac{\tau-k}{3}+\frac{5}{3}b_1+\frac{5}{3}b_2=0.
  \]
\end{corollary}

\noindent
Now we are ready to prove $c(\mathcal{B})\le 0$ for every (finite) prospective $5$-regular $\mathcal{TQ}$-class of a matchstick graph with maximum vertex degree at most $5$. To be more precisely we need to consider finite matchstick graphs $\mathcal{M}$ with maximum vertex degree at most $5$ and a prospective $5$-regular $\mathcal{TQ}$-class $\mathcal{B}$ of $\mathcal{M}$. In $\mathcal{M}$ all vertices of $\mathcal{B}$ must have vertex degree exactly $5$.
In order to be able to speak of a contribution $c(\mathcal{B})$ for every vertex $v$ in $\mathcal{B}$ there must be five faces $(v,i)$, unequal to the outer face, completely being contained in $\mathcal{M}$. If a finite $5$-regular matchstick graph would exist, it would certainly be such a graph. Otherwise we can only look at local parts of such a (possible infinite) graph.

As mentioned in the previous section we have to perform some bookkeeping in order to prove $c\!\left(\cup_i\mathcal{B}_i\right)\le 0$ for a set of prospective $5$-regular $\mathcal{TQ}$-classes. To every face arc $[v,i]$ where $v$ is contained in a given prospective $5$-regular $\mathcal{TQ}$-class $\mathcal{B}$ we will assign a real weight $\omega([v,i])\ge c([v,i])$ which fulfills
\[
  \sum_{v\in\mathcal{B}}\sum_{i=1}^5 \omega([v,i])\le 0.
\]

We start with the face arcs $[v,i]$ where the face $(v,i)$ is contained in $\mathcal{B}$. We call those arcs inner arcs and set $\omega([v,i])=c([v,i])$. With this the sum $\sum\limits_{[v,i]\text{ inner arc of }\mathcal{B}}\omega([v,i])$ over the inner arcs equals the parameter $\sigma$ from Corollary \ref{cor_parameter}.

\begin{figure}[htp]
\begin{center}
\includegraphics{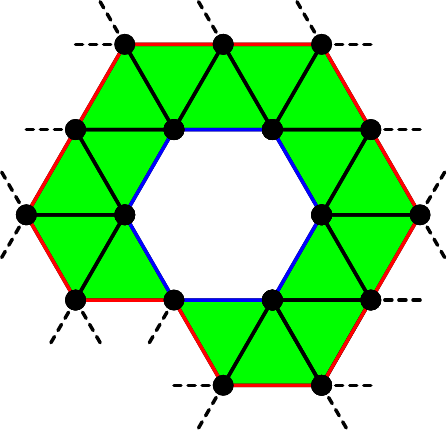}
\caption{A part of a $5$-regular matchstick graph.}
\label{fig_honeycomb_components_3}
\end{center}
\end{figure}

Now we consider the remaining face arcs $[v,i]$ which are given by a vertex $v$ and two edges $e$, $e'$ being adjacent to $v$. We say that the edges $e$ and $e'$ are associated to $[v,i]$.
Let $v$ be an outer vertex of $\mathcal{B}$ and $e$ be an edge (in $\mathcal{M}$) being adjacent to $v$.
If $e$ is not an outer edge of $\mathcal{B}$ we call $e$ a \textit{leaving edge}. This is only possible if $e$ is contained in $\mathcal{M}$ but not in $\mathcal{B}$.
To emphasize that we consider edge $e$ being rooted at vertex $v$ we also speak of \textit{leaving half edges}.
It may happen that a leaving edge $e=\{u,v\}$ corresponds to two leaving half edges being rooted at vertex $u$ and $v$, respectively.
We remark that the number of outer edges equals $k$ and that the number of leaving half edges equals $\tau$ in Corollary \ref{cor_parameter}.

In Figure \ref{fig_honeycomb_components_3} we have depicted the outer edges in blue and red. The leaving half edges are depicted by dashed lines.

If $[v,i]$ is a face arc where both associated edges $e$ and $e'$ correspond to leaving half edges then we set $\omega([v,i])=\frac{1}{3}$ being greater or equal to $c([v,i])$.

For the remaining face arcs we consider the outer edges of $\mathcal{B}$. They form a union of simple cycles, i.~e.{} cycles without repeated vertices. In the example of Figure \ref{fig_honeycomb_components_2} we have a cycle of length $6$ and a cycle of length $15$. We will treat each cycle $C$ separately. Such a cycle $C$ divides the plane into two parts. We call the part containing the faces of $\mathcal{B}$ the interior of $C$ and the other part the exterior of $C$. When we speak of leaving half edges then we only want to address those which go into the exterior of $C$. So the blue cycle of Figure \ref{fig_honeycomb_components_3} does not contain a leaving half edge whereas the red cycle contains $19$ leaving half edges.

It may happen that such a cycle $C$ does not contain any leaving half edges at all, like the blue cycle of Figure \ref{fig_honeycomb_components_3}. In this case $C$ consists of the single face $(v,i)$, where $[v,i]$ is an arbitrary face arc being associated to an edge of $C$. Since $(v,i)$ is not contained in $\mathcal{B}$ it is neither a triangle nor a quadrangle. If $(v,i)$ is a pentagon we set $\omega([v,i])=c([v,i])=-1$ for all associated face arcs $[v,i]$ otherwise we set $\omega([v,i])=-\frac{4}{3}\ge c([v,i])$.

If $C$ does contain leaving half edges the face which is adjacent to a given outer edge $e$ and not contained in $\mathcal{B}$ is bordered by two leaving half edges. In Figure \ref{fig_weights} we have depicted the possible cases.
In principle it would be possible that $C$ contains only one leaving half edge $e$. In such a situation we consider the two half edges the cases of Figure \ref{fig_weights} as being identified to $e$.

\begin{figure}[htp]
\begin{center}
\includegraphics{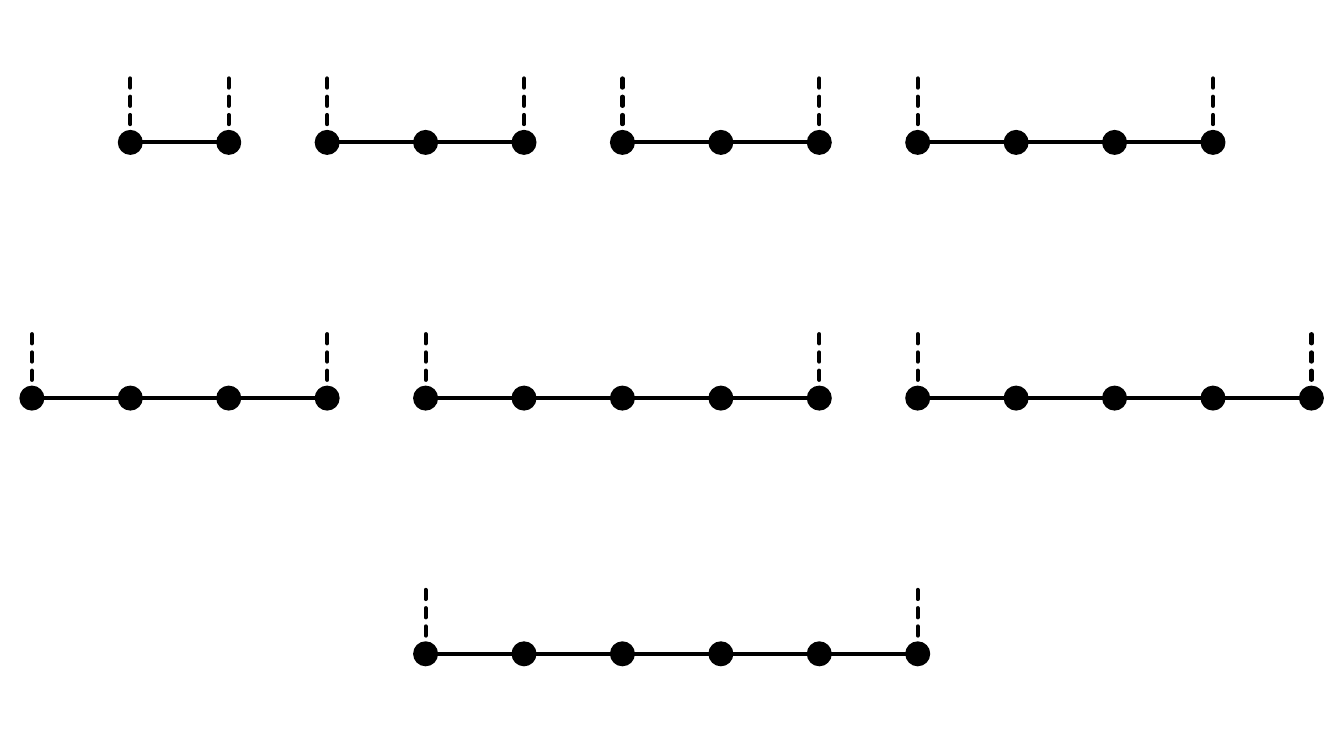}
\caption{Assigning weights.}
\label{fig_weights}
\end{center}
\end{figure}

In case (1) a single outer edge is bordered by two leaving half edges. Since the outer face is not part of $\mathcal{B}$ it cannot be a triangle. So it is at least a quadrangle and we can set $\omega([v,i])=-\frac{1}{2}\ge c([v,i])$. In cases (2a) and (2b) we consider a path $P$ of length two which is bordered by two leaving half edges. Since the outer face is not part of $\mathcal{B}$ it can not be a triangle or a quadrangle. If it is a pentagon we set $\omega([v,i])=-\frac{2}{3}\ge c([v,i])$ for the ends of $P$ and $\omega([v,i])=-1\ge c([v,i])$ for the central vertex. If the outer face has at least $6$ edges we set $\omega([v,i])=-\frac{1}{2}\ge c([v,i])$ for the ends of $P$ and $\omega([v,i])=-\frac{4}{3}\ge c([v,i])$ for the central vertex. In the remaining cases we consider paths of length at least $3$ in case (3a), (3b), paths of length $4$ in case (4a), (4b), and paths of length at least $5$ in case (5). Again the outer face cannot be a triangle or a quadrangle. In case (5) it even cannot be a pentagon. Here we do the assignments of the weights as depicted in Figure \ref{fig_weights}. We can easily check that we have $\omega([v,i])\ge c([v,i])$.

\begin{lemma}
  Let $\mathcal{B}$ be a given prospective $5$-regular $\mathcal{TQ}$-class of matchstick graph $\mathcal{M}$ with
  vertex degree at most $5$. If $c(\mathcal{B})$ can be evaluated within $\mathcal{M}$ then the weight function
  $\omega$ constructed above fulfills
  \begin{equation}\label{eqn_non_negativ}\sum_{v\in\mathcal{B}}\sum_{i=1}^5 \omega([v,i])\le 0.\end{equation}
  and $\omega([v,i])\ge c([v,i])$ for all vertices $v\in\mathcal{B}$, $1\le i\le 5$.
\end{lemma}
\begin{proof}
  By construction we have $\omega([v,i])\ge c([v,i])$ for all $v\in\mathcal{B}$, $1\le i\le 5$. To prove Equation
  (\ref{eqn_non_negativ}) we utilize a booking technique. We want to book face arcs $[v,i]$, outer edges, leaving
  half edges and vertices of special type.

  By $\Omega$ we denote the sum of weights $\omega([v,i])$ over the arcs $[v,i]$
  which are booked. If we book an arc $[v,i]$ we also want to book half of its both associated edges and half edges
  each. Therefore we define $\eta(e)=1$ if edge or half edge $e$ was never booked, $\eta(e)=\frac{1}{2}$ if $e$ was
  booked only one time, and $\eta(e)=0$ if $e$ was booked two times. By $\mathcal{K}$ we denote the sum over $\eta(e)$,
  where the edges $e$ are outer edges of $\mathcal{B}$ and by $T$ we denote the sum over all leaving half edges. By
  $B_1$ we count the number of vertices of special type which are not booked so far.

  As done in the construction of $\omega$ we start with those face arcs $[v,i]$, where the corresponding face $(v,i)$
  is contained in $\mathcal{B}$. Here we set $\omega([v,i])=c([v,i])$. After this initialization we have $\Omega=\sigma$,
  $\mathcal{K}=k$, $T=\tau$, and $B_1=b_1$ using the notation of Lemma \ref{lemma_parameter}. With
  this we have
  \begin{equation}
    \label{eqn_bookkeeping}
    \Omega-\mathcal{K}+\frac{T-\mathcal{K}}{3}+\frac{5}{3}B_1\le 0
  \end{equation}
  since $b_2\ge 0$. Now we book the remaining face arcs and keep track that Inequality (\ref{eqn_bookkeeping}) endures.

  If $[v,i]$ is a face arc between two leaving half edges $e$ and $e'$, then we book these edges and assign
  $\omega([v,i])=\frac{1}{3}$. By this booking step $\Omega$ increases by $\frac{1}{3}$ and $T$ decreases by $1$.
  Thus Inequality (\ref{eqn_bookkeeping}) remains valid.

  Next we consider the simple cycles $C$ of the outer edges. We start with the cases where $C$ does not contain any
  leaving half edges. In this case $C$ consists of the single face $(v,i)$, where $[v,i]$ is an arbitrary face arc
  being associated to an edge of $C$. If $(v,i)$ is a pentagon both $\Omega$ and $\mathcal{K}$ decrease by $5$.
  We can easily check that on a regular triangular grid there does not exist an equilateral pentagon where all
  inner angles are at most $\frac{2\pi}{3}$. Thus face $(v,i)$ does contain an inner angle of at least $\pi$ at a vertex
  $u$. This vertex $u$ must be of special type due to its angle sum of $2\pi$. Thus we can book the special type
  and decrease $B_1$ by one so that Inequality (\ref{eqn_bookkeeping}) remains valid. If $(v,i)$ contains $l\ge 6$
  edges we set $\omega([v,i])=-\frac{4}{3}$ at the corresponding face arcs. Thus $\mathcal{K}$ decreases by $l$ and
  $\Omega$ decreases by $\frac{4}{3}l$ so that Inequality (\ref{eqn_bookkeeping}) remains valid.

  Next we consider the cases of Figure \ref{fig_weights}. In all cases $T$ decreases by one. The decreases of
  $\mathcal{K}$ and $\Omega$ vary in the different cases, but we can easily check that Inequality
  (\ref{eqn_bookkeeping}) remains valid.

  A the end of this procedure we have $\mathcal{K}=T=0$ since all edges are booked properly. Since all vertices
  of special type are booked at most once we have $B_1\ge 0$. Thus we can conclude 
  \[
    \sum_{v\in\mathcal{B}}\sum_{i=1}^5 \omega([v,i])=\Omega\le 0
  \]
  from Inequality (\ref{eqn_bookkeeping}) in the end.
\end{proof}

\begin{corollary}
  For a prospective $5$-regular $\mathcal{TQ}$-class $\mathcal{B}$ of a matchstick graph with maximum
  vertex degree $5$ we have
  \[
    c(\mathcal{B})\le 0.
  \]
\end{corollary}

As depicted in Figure \ref{fig_honeycomb_components} it may happen that a vertex $v$ is contained in more than one prospective $5$-regular $\mathcal{TQ}$-class $\mathcal{B}$. In this cases it belongs to exactly two prospective $5$-regular $\mathcal{TQ}$-classes due to an vertex degree of five. In Figure \ref{fig_two_tq_classes} we have depicted the possible cases. Here one prospective $5$-regular $\mathcal{TQ}$-class is depicted in green and the other one is depicted in blue. In case (1) the angles of face arcs $[v,3]$ and $[v,5]$ need not to be multiples of $\frac{\pi}{3}$. Similarly also in case (2) the angles of the face arcs $[v,2]$, $[v,3]$, and $[v,5]$ need not to be multiples of $\frac{\pi}{3}$. In the next lemma we show that the assignment of our weight functions do not underestimate the contribution of vertex $v$.

\begin{figure}[htp]
\begin{center}
\includegraphics{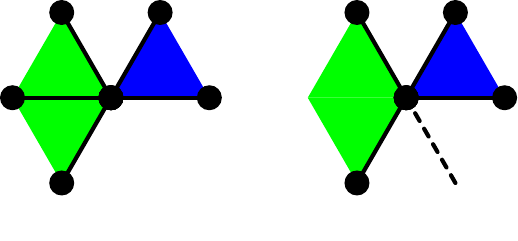}
\caption{Two prospective $5$-regular $\mathcal{TQ}$-classes with a common vertex.}
\label{fig_two_tq_classes}
\end{center}
\end{figure}

\begin{lemma}
  Let $\mathcal{B}_1$ and $\mathcal{B}_2$ be two different prospective $5$-regular $\mathcal{TQ}$-classes of a
  matchstick graph with maximum vertex degree at most $5$ and weight functions $\omega_1$ and $\omega_2$,
  respectively. If a vertex $v$ is contained in both $\mathcal{B}_1$ and $\mathcal{B}_2$ then we have
  $$
    \sum_{i=1}^5 \omega_1([v,i])+\omega_2([v,i])\ge \sum_{i=1}^5 c([v,i]).
  $$
\end{lemma}
\begin{proof} 
  W.l.o.g.\ the prospective $5$-regular $\mathcal{TQ}$-class $\mathcal{B}_1$ is depicted green in Figure
  \ref{fig_two_tq_classes} and $\mathcal{B}_2$ is depicted blue.

  At first we consider case (i). Here we have $\omega_1([v,1])\ge c[(v,1])$, $\omega_1([v,2])\ge c[(v,2])$, and
  $\omega_2([v,4])\ge c([v,4])$. Next we look at the face arcs which are bordered by two leaving half edges. Here we
  have $\omega_2([v,1])=\frac{1}{3}$, $\omega_2([v,2])=\frac{1}{3}$, and $\omega_1([v,4])=\frac{1}{3}$. Due to symmetry
  it suffices to show
  \begin{equation}
    \label{eq_overcount}
    \omega_1([v,3])+\omega_2([v,3])\ge c([v,3])-\frac{1}{2}.
  \end{equation}

  Since at vertex $v$ there are leaving half edges both in $\mathcal{B}_1$ and $\mathcal{B}_2$ we have to be in one
  of the cases of Figure \ref{fig_weights} for the determination of the weight of face arc $[v,3]$. If for
  $\mathcal{B}_1$ face arc $[v,3]$ is in one of the cases (1), (2b), (3b), (4b), or (5) then we have
  $\omega_1([v,3])=-\frac{1}{2}$ and Inequality~(\ref{eq_overcount}) is valid. Due to symmetry the same holds if
  for $\mathcal{B}_2$ face arc $[v,3]$ is in one of the cases (1), (2b), (3b), (4b), or (5). In all other cases
  $(v,i)$ is a pentagon so that $c([v,i])= -1$.
  We remark that $[v,3]$ can not be in case (3a) both in $\mathcal{B}_1$ and $\mathcal{B}_2$. Thus we have
  $\omega_1([v,3])+\omega_2([v,3])\ge -\frac{5}{6}-\frac{2}{3}=-1-\frac{1}{2}\ge c([v,3])-\frac{1}{2}$.

  To finish the proof we consider case (ii). Here we have $\omega_1([v,1])\ge c[(v,1])$, $\omega_1([v,2])\ge c[(v,2])$,
  $\omega_2([v,4])\ge c([v,4])$, $\omega_2([v,3])\ge c([v,3])$, $\omega_1([v,3])=\frac{1}{3}$,
  $\omega_1([v,4])=\frac{1}{3}$,
  $\omega_2([v,1])=\frac{1}{3}$, and $\omega_2([v,2])=\frac{1}{3}$. Thus it suffices to show
  $$
    \omega_1([v,5])+\omega_2([v,5])\ge c([v,5])-\frac{4}{3}.
  $$
  Since $\omega_1([v,5]),\omega_2([v,5])\ge-\frac{5}{6}$ and $c([v,5])\le -\frac{1}{2}$ this inequality is true.
\end{proof}

\begin{corollary}
  \label{cor_non_negativ_all}
  If $\mathcal{B}_1$, $\mathcal{B}_2$, $\dots$ are all prospective $5$-regular$\mathcal{TQ}$-classes of a
  (finite) matchstick graph with maximum vertex degree at most $5$ we have
  $$
    c\!\left(\underset{i}{\cup}\mathcal{B}_i\right)\le 0.
  $$
\end{corollary}

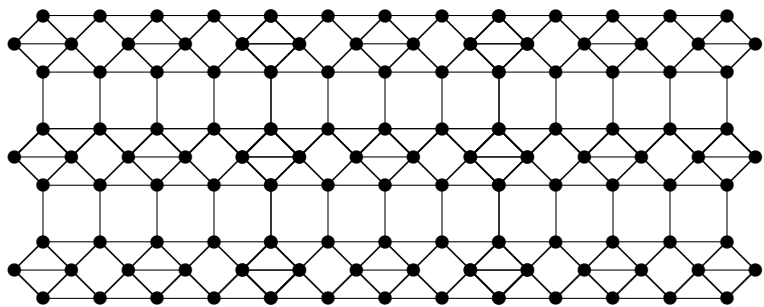
\begin{figure}[ht]
  \begin{center}
    \setlength{\unitlength}{0.75cm}
    \begin{picture}(13,5)
      \put(0.5,0){\line(1,0){4}}
      \put(0,0.5){\line(1,0){1}}
      \put(2,0.5){\line(1,0){1}}
      \put(4,0.5){\line(1,0){1}}
      \put(0.5,1){\line(1,0){4}}
      \put(0.5,2){\line(1,0){4}}
      \put(0,2.5){\line(1,0){1}}
      \put(2,2.5){\line(1,0){1}}
      \put(4,2.5){\line(1,0){1}}
      \put(0.5,3){\line(1,0){4}}
      \put(0.5,0){\line(-1,1){0.5}}
      \put(0.5,0){\line(1,1){0.5}}
      \put(1.5,0){\line(-1,1){0.5}}
      \put(1.5,0){\line(1,1){0.5}}
      \put(2.5,0){\line(-1,1){0.5}}
      \put(2.5,0){\line(1,1){0.5}}
      \put(3.5,0){\line(-1,1){0.5}}
      \put(3.5,0){\line(1,1){0.5}}
      \put(4.5,0){\line(-1,1){0.5}}
      \put(4.5,0){\line(1,1){0.5}}
      \put(0.5,1){\line(-1,-1){0.5}}
      \put(0.5,1){\line(1,-1){0.5}}
      \put(1.5,1){\line(-1,-1){0.5}}
      \put(1.5,1){\line(1,-1){0.5}}
      \put(2.5,1){\line(-1,-1){0.5}}
      \put(2.5,1){\line(1,-1){0.5}}
      \put(3.5,1){\line(-1,-1){0.5}}
      \put(3.5,1){\line(1,-1){0.5}}
      \put(4.5,1){\line(-1,-1){0.5}}
      \put(4.5,1){\line(1,-1){0.5}}
      \put(0.5,2){\line(-1,1){0.5}}
      \put(0.5,2){\line(1,1){0.5}}
      \put(1.5,2){\line(-1,1){0.5}}
      \put(1.5,2){\line(1,1){0.5}}
      \put(2.5,2){\line(-1,1){0.5}}
      \put(2.5,2){\line(1,1){0.5}}
      \put(3.5,2){\line(-1,1){0.5}}
      \put(3.5,2){\line(1,1){0.5}}
      \put(4.5,2){\line(-1,1){0.5}}
      \put(4.5,2){\line(1,1){0.5}}
      \put(0.5,3){\line(-1,-1){0.5}}
      \put(0.5,3){\line(1,-1){0.5}}
      \put(1.5,3){\line(-1,-1){0.5}}
      \put(1.5,3){\line(1,-1){0.5}}
      \put(2.5,3){\line(-1,-1){0.5}}
      \put(2.5,3){\line(1,-1){0.5}}
      \put(3.5,3){\line(-1,-1){0.5}}
      \put(3.5,3){\line(1,-1){0.5}}
      \put(4.5,3){\line(-1,-1){0.5}}
      \put(4.5,3){\line(1,-1){0.5}}
      \put(0.5,1){\line(0,1){1}}
      \put(1.5,1){\line(0,1){1}}
      \put(2.5,1){\line(0,1){1}}
      \put(3.5,1){\line(0,1){1}}
      \put(4.5,1){\line(0,1){1}}
      \put(0.5,0){\circle*{0.25}}
      \put(1.5,0){\circle*{0.25}}
      \put(2.5,0){\circle*{0.25}}
      \put(3.5,0){\circle*{0.25}}
      \put(4.5,0){\circle*{0.25}}
      \put(0.5,1){\circle*{0.25}}
      \put(1.5,1){\circle*{0.25}}
      \put(2.5,1){\circle*{0.25}}
      \put(3.5,1){\circle*{0.25}}
      \put(4.5,1){\circle*{0.25}}
      \put(0.5,2){\circle*{0.25}}
      \put(1.5,2){\circle*{0.25}}
      \put(2.5,2){\circle*{0.25}}
      \put(3.5,2){\circle*{0.25}}
      \put(4.5,2){\circle*{0.25}}
      \put(0.5,3){\circle*{0.25}}
      \put(1.5,3){\circle*{0.25}}
      \put(2.5,3){\circle*{0.25}}
      \put(3.5,3){\circle*{0.25}}
      \put(4.5,3){\circle*{0.25}}
      \put(0,0.5){\circle*{0.25}}
      \put(1,0.5){\circle*{0.25}}
      \put(2,0.5){\circle*{0.25}}
      \put(3,0.5){\circle*{0.25}}
      \put(4,0.5){\circle*{0.25}}
      \put(5,0.5){\circle*{0.25}}
      \put(0,2.5){\circle*{0.25}}
      \put(1,2.5){\circle*{0.25}}
      \put(2,2.5){\circle*{0.25}}
      \put(3,2.5){\circle*{0.25}}
      \put(4,2.5){\circle*{0.25}}
      \put(5,2.5){\circle*{0.25}}
      \put(4.5,0){\line(1,0){4}}
      \put(4,0.5){\line(1,0){1}}
      \put(6,0.5){\line(1,0){1}}
      \put(8,0.5){\line(1,0){1}}
      \put(4.5,1){\line(1,0){4}}
      \put(4.5,2){\line(1,0){4}}
      \put(4,2.5){\line(1,0){1}}
      \put(6,2.5){\line(1,0){1}}
      \put(8,2.5){\line(1,0){1}}
      \put(4.5,3){\line(1,0){4}}
      \put(4.5,0){\line(-1,1){0.5}}
      \put(4.5,0){\line(1,1){0.5}}
      \put(5.5,0){\line(-1,1){0.5}}
      \put(5.5,0){\line(1,1){0.5}}
      \put(6.5,0){\line(-1,1){0.5}}
      \put(6.5,0){\line(1,1){0.5}}
      \put(7.5,0){\line(-1,1){0.5}}
      \put(7.5,0){\line(1,1){0.5}}
      \put(8.5,0){\line(-1,1){0.5}}
      \put(8.5,0){\line(1,1){0.5}}
      \put(4.5,1){\line(-1,-1){0.5}}
      \put(4.5,1){\line(1,-1){0.5}}
      \put(5.5,1){\line(-1,-1){0.5}}
      \put(5.5,1){\line(1,-1){0.5}}
      \put(6.5,1){\line(-1,-1){0.5}}
      \put(6.5,1){\line(1,-1){0.5}}
      \put(7.5,1){\line(-1,-1){0.5}}
      \put(7.5,1){\line(1,-1){0.5}}
      \put(8.5,1){\line(-1,-1){0.5}}
      \put(8.5,1){\line(1,-1){0.5}}
      \put(4.5,2){\line(-1,1){0.5}}
      \put(4.5,2){\line(1,1){0.5}}
      \put(5.5,2){\line(-1,1){0.5}}
      \put(5.5,2){\line(1,1){0.5}}
      \put(6.5,2){\line(-1,1){0.5}}
      \put(6.5,2){\line(1,1){0.5}}
      \put(7.5,2){\line(-1,1){0.5}}
      \put(7.5,2){\line(1,1){0.5}}
      \put(8.5,2){\line(-1,1){0.5}}
      \put(8.5,2){\line(1,1){0.5}}
      \put(4.5,3){\line(-1,-1){0.5}}
      \put(4.5,3){\line(1,-1){0.5}}
      \put(5.5,3){\line(-1,-1){0.5}}
      \put(5.5,3){\line(1,-1){0.5}}
      \put(6.5,3){\line(-1,-1){0.5}}
      \put(6.5,3){\line(1,-1){0.5}}
      \put(7.5,3){\line(-1,-1){0.5}}
      \put(7.5,3){\line(1,-1){0.5}}
      \put(8.5,3){\line(-1,-1){0.5}}
      \put(8.5,3){\line(1,-1){0.5}}
      \put(4.5,1){\line(0,1){1}}
      \put(5.5,1){\line(0,1){1}}
      \put(6.5,1){\line(0,1){1}}
      \put(7.5,1){\line(0,1){1}}
      \put(8.5,1){\line(0,1){1}}
      \put(4.5,0){\circle*{0.25}}
      \put(5.5,0){\circle*{0.25}}
      \put(6.5,0){\circle*{0.25}}
      \put(7.5,0){\circle*{0.25}}
      \put(8.5,0){\circle*{0.25}}
      \put(4.5,1){\circle*{0.25}}
      \put(5.5,1){\circle*{0.25}}
      \put(6.5,1){\circle*{0.25}}
      \put(7.5,1){\circle*{0.25}}
      \put(8.5,1){\circle*{0.25}}
      \put(4.5,2){\circle*{0.25}}
      \put(5.5,2){\circle*{0.25}}
      \put(6.5,2){\circle*{0.25}}
      \put(7.5,2){\circle*{0.25}}
      \put(8.5,2){\circle*{0.25}}
      \put(4.5,3){\circle*{0.25}}
      \put(5.5,3){\circle*{0.25}}
      \put(6.5,3){\circle*{0.25}}
      \put(7.5,3){\circle*{0.25}}
      \put(8.5,3){\circle*{0.25}}
      \put(4,0.5){\circle*{0.25}}
      \put(5,0.5){\circle*{0.25}}
      \put(6,0.5){\circle*{0.25}}
      \put(7,0.5){\circle*{0.25}}
      \put(8,0.5){\circle*{0.25}}
      \put(9,0.5){\circle*{0.25}}
      \put(4,2.5){\circle*{0.25}}
      \put(5,2.5){\circle*{0.25}}
      \put(6,2.5){\circle*{0.25}}
      \put(7,2.5){\circle*{0.25}}
      \put(8,2.5){\circle*{0.25}}
      \put(9,2.5){\circle*{0.25}}
      \put(8.5,0){\line(1,0){4}}
      \put(8,0.5){\line(1,0){1}}
      \put(10,0.5){\line(1,0){1}}
      \put(12,0.5){\line(1,0){1}}
      \put(8.5,1){\line(1,0){4}}
      \put(8.5,2){\line(1,0){4}}
      \put(8,2.5){\line(1,0){1}}
      \put(10,2.5){\line(1,0){1}}
      \put(12,2.5){\line(1,0){1}}
      \put(8.5,3){\line(1,0){4}}
      \put(8.5,0){\line(-1,1){0.5}}
      \put(8.5,0){\line(1,1){0.5}}
      \put(9.5,0){\line(-1,1){0.5}}
      \put(9.5,0){\line(1,1){0.5}}
      \put(10.5,0){\line(-1,1){0.5}}
      \put(10.5,0){\line(1,1){0.5}}
      \put(11.5,0){\line(-1,1){0.5}}
      \put(11.5,0){\line(1,1){0.5}}
      \put(12.5,0){\line(-1,1){0.5}}
      \put(12.5,0){\line(1,1){0.5}}
      \put(8.5,1){\line(-1,-1){0.5}}
      \put(8.5,1){\line(1,-1){0.5}}
      \put(9.5,1){\line(-1,-1){0.5}}
      \put(9.5,1){\line(1,-1){0.5}}
      \put(10.5,1){\line(-1,-1){0.5}}
      \put(10.5,1){\line(1,-1){0.5}}
      \put(11.5,1){\line(-1,-1){0.5}}
      \put(11.5,1){\line(1,-1){0.5}}
      \put(12.5,1){\line(-1,-1){0.5}}
      \put(12.5,1){\line(1,-1){0.5}}
      \put(8.5,2){\line(-1,1){0.5}}
      \put(8.5,2){\line(1,1){0.5}}
      \put(9.5,2){\line(-1,1){0.5}}
      \put(9.5,2){\line(1,1){0.5}}
      \put(10.5,2){\line(-1,1){0.5}}
      \put(10.5,2){\line(1,1){0.5}}
      \put(11.5,2){\line(-1,1){0.5}}
      \put(11.5,2){\line(1,1){0.5}}
      \put(12.5,2){\line(-1,1){0.5}}
      \put(12.5,2){\line(1,1){0.5}}
      \put(8.5,3){\line(-1,-1){0.5}}
      \put(8.5,3){\line(1,-1){0.5}}
      \put(9.5,3){\line(-1,-1){0.5}}
      \put(9.5,3){\line(1,-1){0.5}}
      \put(10.5,3){\line(-1,-1){0.5}}
      \put(10.5,3){\line(1,-1){0.5}}
      \put(11.5,3){\line(-1,-1){0.5}}
      \put(11.5,3){\line(1,-1){0.5}}
      \put(12.5,3){\line(-1,-1){0.5}}
      \put(12.5,3){\line(1,-1){0.5}}
      \put(8.5,1){\line(0,1){1}}
      \put(9.5,1){\line(0,1){1}}
      \put(10.5,1){\line(0,1){1}}
      \put(11.5,1){\line(0,1){1}}
      \put(12.5,1){\line(0,1){1}}
      \put(8.5,0){\circle*{0.25}}
      \put(9.5,0){\circle*{0.25}}
      \put(10.5,0){\circle*{0.25}}
      \put(11.5,0){\circle*{0.25}}
      \put(12.5,0){\circle*{0.25}}
      \put(8.5,1){\circle*{0.25}}
      \put(9.5,1){\circle*{0.25}}
      \put(10.5,1){\circle*{0.25}}
      \put(11.5,1){\circle*{0.25}}
      \put(12.5,1){\circle*{0.25}}
      \put(8.5,2){\circle*{0.25}}
      \put(9.5,2){\circle*{0.25}}
      \put(10.5,2){\circle*{0.25}}
      \put(11.5,2){\circle*{0.25}}
      \put(12.5,2){\circle*{0.25}}
      \put(8.5,3){\circle*{0.25}}
      \put(9.5,3){\circle*{0.25}}
      \put(10.5,3){\circle*{0.25}}
      \put(11.5,3){\circle*{0.25}}
      \put(12.5,3){\circle*{0.25}}
      \put(8,0.5){\circle*{0.25}}
      \put(9,0.5){\circle*{0.25}}
      \put(10,0.5){\circle*{0.25}}
      \put(11,0.5){\circle*{0.25}}
      \put(12,0.5){\circle*{0.25}}
      \put(13,0.5){\circle*{0.25}}
      \put(8,2.5){\circle*{0.25}}
      \put(9,2.5){\circle*{0.25}}
      \put(10,2.5){\circle*{0.25}}
      \put(11,2.5){\circle*{0.25}}
      \put(12,2.5){\circle*{0.25}}
      \put(13,2.5){\circle*{0.25}}
      \put(0.5,3){\line(1,0){4}}
      \put(0.5,4){\line(1,0){4}}
      \put(0,4.5){\line(1,0){1}}
      \put(2,4.5){\line(1,0){1}}
      \put(4,4.5){\line(1,0){1}}
      \put(0.5,5){\line(1,0){4}}
      \put(0.5,3){\line(-1,-1){0.5}}
      \put(0.5,3){\line(1,-1){0.5}}
      \put(1.5,3){\line(-1,-1){0.5}}
      \put(1.5,3){\line(1,-1){0.5}}
      \put(2.5,3){\line(-1,-1){0.5}}
      \put(2.5,3){\line(1,-1){0.5}}
      \put(3.5,3){\line(-1,-1){0.5}}
      \put(3.5,3){\line(1,-1){0.5}}
      \put(4.5,3){\line(-1,-1){0.5}}
      \put(4.5,3){\line(1,-1){0.5}}
      \put(0.5,4){\line(-1,1){0.5}}
      \put(0.5,4){\line(1,1){0.5}}
      \put(1.5,4){\line(-1,1){0.5}}
      \put(1.5,4){\line(1,1){0.5}}
      \put(2.5,4){\line(-1,1){0.5}}
      \put(2.5,4){\line(1,1){0.5}}
      \put(3.5,4){\line(-1,1){0.5}}
      \put(3.5,4){\line(1,1){0.5}}
      \put(4.5,4){\line(-1,1){0.5}}
      \put(4.5,4){\line(1,1){0.5}}
      \put(0.5,5){\line(-1,-1){0.5}}
      \put(0.5,5){\line(1,-1){0.5}}
      \put(1.5,5){\line(-1,-1){0.5}}
      \put(1.5,5){\line(1,-1){0.5}}
      \put(2.5,5){\line(-1,-1){0.5}}
      \put(2.5,5){\line(1,-1){0.5}}
      \put(3.5,5){\line(-1,-1){0.5}}
      \put(3.5,5){\line(1,-1){0.5}}
      \put(4.5,5){\line(-1,-1){0.5}}
      \put(4.5,5){\line(1,-1){0.5}}
      \put(0.5,3){\line(0,1){1}}
      \put(1.5,3){\line(0,1){1}}
      \put(2.5,3){\line(0,1){1}}
      \put(3.5,3){\line(0,1){1}}
      \put(4.5,3){\line(0,1){1}}
      \put(0.5,4){\circle*{0.25}}
      \put(1.5,4){\circle*{0.25}}
      \put(2.5,4){\circle*{0.25}}
      \put(3.5,4){\circle*{0.25}}
      \put(4.5,4){\circle*{0.25}}
      \put(0.5,5){\circle*{0.25}}
      \put(1.5,5){\circle*{0.25}}
      \put(2.5,5){\circle*{0.25}}
      \put(3.5,5){\circle*{0.25}}
      \put(4.5,5){\circle*{0.25}}
      \put(0,4.5){\circle*{0.25}}
      \put(1,4.5){\circle*{0.25}}
      \put(2,4.5){\circle*{0.25}}
      \put(3,4.5){\circle*{0.25}}
      \put(4,4.5){\circle*{0.25}}
      \put(5,4.5){\circle*{0.25}}
      \put(4.5,3){\line(1,0){4}}
      \put(4.5,4){\line(1,0){4}}
      \put(4,4.5){\line(1,0){1}}
      \put(6,4.5){\line(1,0){1}}
      \put(8,4.5){\line(1,0){1}}
      \put(4.5,5){\line(1,0){4}}
      \put(4.5,3){\line(-1,-1){0.5}}
      \put(4.5,3){\line(1,-1){0.5}}
      \put(5.5,3){\line(-1,-1){0.5}}
      \put(5.5,3){\line(1,-1){0.5}}
      \put(6.5,3){\line(-1,-1){0.5}}
      \put(6.5,3){\line(1,-1){0.5}}
      \put(7.5,3){\line(-1,-1){0.5}}
      \put(7.5,3){\line(1,-1){0.5}}
      \put(8.5,3){\line(-1,-1){0.5}}
      \put(8.5,3){\line(1,-1){0.5}}
      \put(4.5,4){\line(-1,1){0.5}}
      \put(4.5,4){\line(1,1){0.5}}
      \put(5.5,4){\line(-1,1){0.5}}
      \put(5.5,4){\line(1,1){0.5}}
      \put(6.5,4){\line(-1,1){0.5}}
      \put(6.5,4){\line(1,1){0.5}}
      \put(7.5,4){\line(-1,1){0.5}}
      \put(7.5,4){\line(1,1){0.5}}
      \put(8.5,4){\line(-1,1){0.5}}
      \put(8.5,4){\line(1,1){0.5}}
      \put(4.5,5){\line(-1,-1){0.5}}
      \put(4.5,5){\line(1,-1){0.5}}
      \put(5.5,5){\line(-1,-1){0.5}}
      \put(5.5,5){\line(1,-1){0.5}}
      \put(6.5,5){\line(-1,-1){0.5}}
      \put(6.5,5){\line(1,-1){0.5}}
      \put(7.5,5){\line(-1,-1){0.5}}
      \put(7.5,5){\line(1,-1){0.5}}
      \put(8.5,5){\line(-1,-1){0.5}}
      \put(8.5,5){\line(1,-1){0.5}}
      \put(4.5,3){\line(0,1){1}}
      \put(5.5,3){\line(0,1){1}}
      \put(6.5,3){\line(0,1){1}}
      \put(7.5,3){\line(0,1){1}}
      \put(8.5,3){\line(0,1){1}}
      \put(4.5,4){\circle*{0.25}}
      \put(5.5,4){\circle*{0.25}}
      \put(6.5,4){\circle*{0.25}}
      \put(7.5,4){\circle*{0.25}}
      \put(8.5,4){\circle*{0.25}}
      \put(4.5,5){\circle*{0.25}}
      \put(5.5,5){\circle*{0.25}}
      \put(6.5,5){\circle*{0.25}}
      \put(7.5,5){\circle*{0.25}}
      \put(8.5,5){\circle*{0.25}}
      \put(4,4.5){\circle*{0.25}}
      \put(5,4.5){\circle*{0.25}}
      \put(6,4.5){\circle*{0.25}}
      \put(7,4.5){\circle*{0.25}}
      \put(8,4.5){\circle*{0.25}}
      \put(9,4.5){\circle*{0.25}}
      \put(8.5,3){\line(1,0){4}}
      \put(8.5,4){\line(1,0){4}}
      \put(8,4.5){\line(1,0){1}}
      \put(10,4.5){\line(1,0){1}}
      \put(12,4.5){\line(1,0){1}}
      \put(8.5,5){\line(1,0){4}}
      \put(8.5,3){\line(-1,-1){0.5}}
      \put(8.5,3){\line(1,-1){0.5}}
      \put(9.5,3){\line(-1,-1){0.5}}
      \put(9.5,3){\line(1,-1){0.5}}
      \put(10.5,3){\line(-1,-1){0.5}}
      \put(10.5,3){\line(1,-1){0.5}}
      \put(11.5,3){\line(-1,-1){0.5}}
      \put(11.5,3){\line(1,-1){0.5}}
      \put(12.5,3){\line(-1,-1){0.5}}
      \put(12.5,3){\line(1,-1){0.5}}
      \put(8.5,4){\line(-1,1){0.5}}
      \put(8.5,4){\line(1,1){0.5}}
      \put(9.5,4){\line(-1,1){0.5}}
      \put(9.5,4){\line(1,1){0.5}}
      \put(10.5,4){\line(-1,1){0.5}}
      \put(10.5,4){\line(1,1){0.5}}
      \put(11.5,4){\line(-1,1){0.5}}
      \put(11.5,4){\line(1,1){0.5}}
      \put(12.5,4){\line(-1,1){0.5}}
      \put(12.5,4){\line(1,1){0.5}}
      \put(8.5,5){\line(-1,-1){0.5}}
      \put(8.5,5){\line(1,-1){0.5}}
      \put(9.5,5){\line(-1,-1){0.5}}
      \put(9.5,5){\line(1,-1){0.5}}
      \put(10.5,5){\line(-1,-1){0.5}}
      \put(10.5,5){\line(1,-1){0.5}}
      \put(11.5,5){\line(-1,-1){0.5}}
      \put(11.5,5){\line(1,-1){0.5}}
      \put(12.5,5){\line(-1,-1){0.5}}
      \put(12.5,5){\line(1,-1){0.5}}
      \put(8.5,3){\line(0,1){1}}
      \put(9.5,3){\line(0,1){1}}
      \put(10.5,3){\line(0,1){1}}
      \put(11.5,3){\line(0,1){1}}
      \put(12.5,3){\line(0,1){1}}
      \put(8.5,4){\circle*{0.25}}
      \put(9.5,4){\circle*{0.25}}
      \put(10.5,4){\circle*{0.25}}
      \put(11.5,4){\circle*{0.25}}
      \put(12.5,4){\circle*{0.25}}
      \put(8.5,5){\circle*{0.25}}
      \put(9.5,5){\circle*{0.25}}
      \put(10.5,5){\circle*{0.25}}
      \put(11.5,5){\circle*{0.25}}
      \put(12.5,5){\circle*{0.25}}
      \put(8,4.5){\circle*{0.25}}
      \put(9,4.5){\circle*{0.25}}
      \put(10,4.5){\circle*{0.25}}
      \put(11,4.5){\circle*{0.25}}
      \put(12,4.5){\circle*{0.25}}
      \put(13,4.5){\circle*{0.25}}
    \end{picture}
  \end{center}
  \caption{Infinite $5$-regular match stick graph.}
  \label{fig_infinite}
\end{figure}

\section{Conclusion}
\noindent
In this paper we have proven that no finite $5$-regular matchstick graph exists. In Figure \ref{fig_infinite} we have depicted a fraction of an infinite $5$-regular matchstick graph. Since we can shift single rows of such a construction there exists an uncountable number of these graphs (even in the combinatorial sense). As mentioned by Bojan Mohar there are several further examples. 
Starting with the infinite $6$-regular triangulation we can delete several vertices at different rows and obtain an example by suitably removing horizontal edges to enforce degree $5$ for all vertices. Taking an arbitrary $\mathcal{TQ}$-class and continueing the  boundary vertices with trees gives another set of examples.

In the context of regular matchstick graphs the only remaining question is the determination of $m(4)$, i~e. the smallest $4$-regular matchstick graph.

We strongly believe that there exists a more general topological interpretation of Equation~(\ref{eqn_paramater}). We have simply discovered it going along the concept of a potential function in order to prove Theorem~\ref{thm_main}.

\providecommand{\href}[2]{#2}

\end{document}